\documentclass[12 pt,reqno,oneside]{amsart}
\usepackage{amsmath,amsthm,amsfonts,amssymb}
\usepackage[mathscr]{eucal}
\usepackage[usenames, dvipsnames]{color}
\usepackage{indentfirst}
\usepackage{graphicx}
\usepackage{url}

\usepackage{setspace} 

\theoremstyle{plain}
\newtheorem{teo}{Theorem}[section]
\newtheorem{cor}[teo]{Corollary}

\newtheorem{prop}[teo]{Proposition}

\theoremstyle{definition}
\newtheorem{defn}[teo]{Definition}
\newtheorem{exa}[teo]{Example}

\numberwithin{equation}{section}

\def\bbP{{\mathbb P}}

\def\bbT{{\mathbb T}}

\def\proc{{\{{\mathbb G}; N\}}}
\def\procT{{\{\bbT^d; N\}}}
\def\procc{{\{\mathbb{G}; \mathcal{C}, \mathcal{E},T \}}}
\def\proccT{{\{\bbT^d; \mathcal{C}, \mathcal{E},T \}}}
\def\procpg{{\{\bbT^d; \mathcal{P}(\lambda), \mathcal{G}(p) \}}}
\def\procyb{{\{\bbT^d; \mathcal{Y}(\lambda), \mathcal{B}(p) \}}}
\def\yb4{{\{\bbT^4; \mathcal{Y}(\lambda), \mathcal{B}(p) \}}}
\def\pgd10{{\{\bbT^{10}; \mathcal{P}(5), \mathcal{G}(0.6) \}}}
\def\pgd{{\{\bbT^{d}; \mathcal{P}(5), \mathcal{G}(0.6) \}}}
\def\up{{U\{\bbT^d; N\}}}
\def\upP{{U\{\bbT^d_+; N\}}}
\def\upPd{{U\{\bbT^{d+1}_+; N\}}}
\def\low{{L\{\bbT^d; N\}}}
\def\lowP{{L\{\bbT^d_+; N\}}}

\begin{document}

\baselineskip=26pt

\address[F\'abio~P.~Machado]
{Institute of Mathematics and Statistics
\\ University of S\~ao Paulo \\ Rua do Mat\~ao 1010, CEP
05508-090, S\~ao Paulo, SP, Brazil.}

\address[Valdivino~V.~Junior]
{Federal University of Goias
\\ Campus Samambaia, CEP 74001-970, Goi\^ania, GO, Brazil.}

\address[Alejandro Roldan-Correa]
{Instituto de Matem\'aticas, Universidad de Antioquia, Calle 67, no 53-108, Medellin, Colombia}

\title[Colonization and collapse on Homogeneous Trees]{Colonization and collapse on Homogeneous Trees}
\author{Valdivino~V.~Junior}
\author{F\'abio~P.~Machado}
\author{Alejandro Rold\'an-Correa}

\noindent
\email{vvjunior@ufg.br, fmachado@ime.usp.br, alejandro.roldan@udea.edu.co}

\thanks{Research supported by CNPq (306927/2007-1), FAPESP (2010/50884-4) and Universidad de Antioquia (SUI XXXXX).}

\keywords{Branching processes, Coupling, Catastrophes, Population dynamics.}

\subjclass[2010]{60J80, 60J85, 92D25}

\date{\today}

\begin{abstract}

We investigate a basic immigration process where co\-lo\-nies grow, during a random time, 
according to a general counting process until collapse. Upon collapse a random amount of 
indivi\-duals survive. These survivors try independently establishing new colonies at neighbour sites.
Here we consider this general process subject to two schemes, Poisson growth with geometric 
catastrophe and Yule growth with binomial catastrophe. Independent of everything else 
colonies growth, during an exponential time, as a Poisson (or Yule) process and right 
after that  exponential time their size is reduced according to geometric (or binomial) law. 
Each survivor tries independently, to start a new colony at a neighbour 
site of a homogeneous tree. That colony will thrive until its collapse, and so on.
We study conditions on the set of parameters for these processes
to survive, present relevant bounds for the probability of survival, for the number of vertices that were
colonized and for the reach of the colonies compared to the starting point.
\end{abstract}

\singlespacing

\maketitle
\section{Introduction}
\label{S: Introduction}
Biological populations are subject to disasters that can cause from a partial elimination of the individuals until their total extinction. When a disaster occurs surviving individuals may react in different ways. A strategy adopted by some populations is the dispersion. In this case, individuals migrate, trying to create new colonies in other locations, there may be competition or collaboration between individuals of the same colony. Once they settle down a new colony in a new spot, again another disaster can strike, which causes a new collapse.

In this type of population dynamics there are some issues to consider, such as: What is the duration of colonization until the moment of the disaster? How much the population grows until be hit? How many individuals will survive? How survivors react when facing a disaster?

In recent articles, the main variables considered in population modeling are (i) the spatial structure where the colonies are located and individuals can move, (ii) the lifetime of a colony until the moment of collapse, (iii) the evolution of the number of individuals in the colony (random or deterministic growth, possible deaths or migration), (iv) the way the cathastrophes affects the size of the colony allowing or not the survival of some individuals and (v) whether the individuals that survive to the catastrophe are able to spread out.

Brockwell \textit{et al.}~\cite{BGR1982} and later Artalejo \textit{et al.}~\cite{AEL2007} considered a model for the growth of a population subject to collapse. In their model, two types of effects when a disaster strikes are analyzed separately, \textit{binomial effect} and \textit{geometric effect}. After the collapse, the survivors remain together in the same colony (there is no dispersion). They carried out an extensive analysis including first extinction time, number of individuals removed, survival time of a tagged individual, and maximum population size reached between two consecutive extinctions.

More recently, Schinazi~\cite{S2014} and Machado \textit{et al.}~\cite{MRS2015} proposed stochastic models for this kind of population dynamics. For these models they concluded that dispersion is a good survival strategy. Latter Junior \textit{et al.}~\cite{VAF2016} showed nice combinations of a type of catastrophe, spatial restriction and individual survival probability when facing the
catastrophe where dispersion may not be a good strategy for group survival. For a
comprehensive literature overview and motivation see Kapodistria \textit{et al.}~\cite{KPR2016}.

The paper is divided into four sections. In Section 2 we present a general model for the growth of populations subject to collapses,  introduce the variables of interest, notation and two particular schemes: Poisson growth with geometric 
catastrophe and Yule growth with binomial catastrophe. In Section 3 we present the main results of the paper while their proofs are in Section 4.

\section{Colonization and Collapse models}
\label{S: CCM}

In the beginning all vertices of $\mathbb{G}$, a infinite conected graph, are empty except for the origin
where there is one individual. Besides that, at any time each colony is started by a single
individual. The number of individuals in each colony behaves as $\mathcal{C}$, a
Counting Process. To each colony is associated a non-negative random
variable $T$ which defines its lifetime. After a period of time $T$, that colony collapses
and the vertex where it is placed becomes empty. At the time of collapse, with a random effect
$\mathcal{E}$, some individuals in the colony are able to survive while others die. By simplicity we represent this quantity by $N$. Note that this random
quantity depends on the Counting Process which defines the growth of the colony, on the distribution
of $T$ and on how the collapse afects the group of individuals present in the colony at time $T$.
Each one individual that survives ($N$ individuals) tries to found a new
colony on one of the nearest neighbour vertices by first picking one of them at random. If the chosen vertex is occupied, that individual dies, otherwise the individual founds there a new colony. We denote the Colonization and Collapse model generally described here either by $\proc$ or $\procc$, a stochastic process whose state space is $\mathbb{N}^{\mathbb{T}^d}$. Along this
paper we concentrate our attention on $\bbT^d$, a homogeneous tree where every vertex has $d+1$ nearest neighbours and on $\bbT^d_+$, a tree whose only difference from $\bbT^d$ is that its origin has degree $d$.

\begin{defn} Let us consider the following random variables
\begin{itemize}
\item $I_d:$ the number of colonies created from the beginning to the end of the process;
\item $M_d:$ the distance from the origin to the furthest vertex where a colony is created;
\item $\{X_t\}_{0 \le t \le T}$  growth process for the amount of individuals in a colony.
\end{itemize}
\end{defn}

We work in details some specific cases.

\begin{itemize}
\item $T:$ Lifetime of a colony
\begin{itemize}
\item $T \sim {\mathcal Exp}(1)$, Exponential with mean 1
\[ P[T<t] = 1- e^{- t}, \ t > 0 \]
\end{itemize}
\item $X_t:$ Growth of the number of individuals
\begin{itemize}
\item $X_t \sim {\mathcal Poisson}(\lambda t)$, a Poisson point process with rate $\lambda$
\[P[X_t = k] = \frac{e^{-\lambda} \lambda^{k-1}}{(k-1)!} , \ k \in \{1,2,...\} \]
\item $X_t \sim {\mathcal Geom}(e^{-\lambda t})$, a Yule process with rate $\lambda$
\[P[X_t = k] = e^{-\lambda t} (1-e^{-\lambda t})^{k-1}  , \ k \in \{1,2,...\} \]
\end{itemize}
\item $N:$ Number of individuals able to survive
\begin{itemize}
\item $N | X_T \sim {\mathcal B}(X_T,p),$ Binomial catastrophe
\[P[N = m | X_T = k] = {k \choose m} p^m(1-p)^{k-m} , \ m \in \{0, 1,...,k\} \]
\item $N | X_T = X_T- \min\{{\mathcal Geom}(p)-1;X_T\} \sim {\mathcal G}_{X_T}(p),$ Geometric catastrophe
\[ P[N = m | X_T] = \left\{ \begin{array}{ll}
p (1-p)^{X_T - m} & \mbox{if $m \in \{1,...,X_T$\}} \\ \\
(1-p)^{X_T} & \mbox{if $m = 0$}
\end{array}
\right. \]
\end{itemize}
\end{itemize}

In general it is true that
\begin{align*}
\mathbb{P}(N = n) &= \int_0^{\infty} \mathbb{P}(N=n|T=t)f_T(t)dt \\
\mathbb{P}(N = n|T=t) &= \sum_{x=n}^{\infty} \mathbb{P}(X_T=x|T=t)\mathbb{P}(N=n|X_T=x ; T=t).
\end{align*}

Suppose that individuals are born following a Poisson process at rate $\lambda$, that
the collapse time follows an exponential random variable with average 1 ($T \sim {\mathcal Exp}(1)$) and the
individuals are exposed to the collapse effects, one by one, until
the first individual survive, if any, then the collapse effects stop. If the collapse effects
reach a fixed individual, it survives with probability $p$,
meaning that $N_T \sim {\mathcal G}_{X_T}(p)$ (Geometric catastrophe) or ${\mathcal G}(p)$
for short. Let us consider the distribution of the number of survivals at collapse times
\begin{align*}
\mathbb{P}(N = 0) = \int_{0}^{\infty}e^{-t}\sum_{j=0}^{\infty}\frac
{e^{-\lambda t}(\lambda t)^j}{j!}(1-p)^{j+1}dt = \frac{1-p}{1 +\lambda
p }
\end{align*}
and for $n \geq 1$:
\begin{align*}
\mathbb{P}(N = n) = \int_{0}^{\infty}e^{-t}\sum_{j=n-1}^
{\infty}\frac{e^{-\lambda t}(\lambda t)^j}{j!}  p(1-p)^{j+1-n}dt
= \left(\frac{\lambda }{\lambda  + 1}\right )^{n-1} \frac{p}{\lambda p +1}.
\end{align*}

\noindent
In this case the probability generating function is of $N$ is
\begin{align}
\label{eq: fgpP}
\mathbb{E}(s^N) =& \frac{1-p}{1+\lambda p}+\sum_{n=1}^{\infty} s^n \left(\frac{\lambda }{1+\lambda }\right)^{n-1} \left(\frac{p}{\lambda p+1}\right) \\
=& \frac{1}{\lambda p +1}\left[1-p+\frac{(\lambda+1)ps}{1+\lambda-\lambda s}\right].
\end{align}
while its average is $\displaystyle \mathbb{E}(N) = \frac{p (\lambda +1)^2}{(\lambda p + 1)}.$

Suppose now that individuals are born following a Yule process at rate $\lambda$, that
$T \sim {\mathcal Exp}(1)$ and that the
disaster reach the individuals simultaneously and independently of everything else.
Assuming that each individual survives with probability $p$,
we have that $N_T \sim {\mathcal B}(X_T, p)$ (Binomial catastrophe) or
${\mathcal B}(p)$ for short. Let us consider the distribution of the number of survivals at
collapse times.
\begin{align*}
\mathbb{P}(N = 0) =&\int_0^\infty e^{-t}\sum_{j= 1}^\infty e^{-\lambda t}(1-e^{-\lambda t})^{j-1}(1-p)^{j}dt \\
=&\frac{1-p}{\lambda+1}\ {_2F_1}\left(1,1;2+\frac{1}{\lambda};{1-p}\right).
\end{align*}
\noindent
and for $n \geq 1$
\begin{align*}
\mathbb{P}(N = n) =& \int_0^\infty e^{-t}\sum_{j= n}^\infty e^{-\lambda t}(1-e^{-\lambda t})^{j-1}{j \choose n}p^n(1-p)^{j-n}dt \\
=&\frac{p^k}{\lambda}B\left(k,1+\frac{1}{\lambda}\right){_2F_1}\left(k+1,k;k+1+\frac{1}{\lambda};{1-p}\right).
\end{align*}
In this setup the probability generating function of $N$ is
\begin{eqnarray}\label{eq: fgpY}
\mathbb{E}(s^N)&=&\sum_{n=0}^{\infty} s^n \int_{0}^{\infty} e^{-t} \sum_{k=n \vee 1}^{\infty} e^{-\lambda t}(1-e^{-\lambda t})^{k-1} {k \choose n} p^n(1-p)^{k-n} \ dt\nonumber \\
&=&\frac{ps+1-p}{\lambda +1}\ {_2F_1}\left(1,1;2+\frac{1}{\lambda}; p(s-1)+1\right)
\end{eqnarray}
and its average is
\begin{eqnarray}\label{eq: averageN}
\mathbb{E}(N)=
\left\{\begin{array}{cl} \displaystyle\frac{p}{1-\lambda} & ,\text{ se }  \lambda<1 \\ \\ \infty &, \text{ se } \lambda\geq 1.
\end{array}\right.\nonumber
\end{eqnarray}

\section{Main Results}
\label{S: Homogeneous Trees}

$\{ \bbT^d, \mathcal{C}, \mathcal{E}, T \}$ is a stochastic process whose state space is $\mathbb{N}^{\mathbb{T}^d}$ and whose evolution (status at time $t$) is denoted by $\eta_t$. For a vertex $x \in \mathbb{T}^d$, $\{\eta_t(x)=i\}$ means that at the time $t$ there are $i$
individuals at the vertex $x$. We consider $|\eta_t| = \sum_{x \in \mathbb{T}^d} \eta_t(x)$.

\subsection{Phase Transition}
\label{SS: PT}

\begin{defn}
Let $\eta_t$ be the process $\proccT$. Let us define the event
\[ V_d =  \{ |\eta_t| > 0, \hbox{ for all } t \ge 0 \}. \]

If $\mathbb{P}(V_d) > 0$ we say that  the process $\proccT$ {\it survives}. Otherwise, we say that the process $\proccT$ {\it dies out }.
\end{defn}

\begin{teo}
\label{T: MCC1H}
Consider the process $\procT$. Then $\mathbb{P}(V_d) = 0 $ if
\begin{displaymath}
\mathbb{E} \left [ \left(\frac{d}{d+1} \right)^N \right] \geq \frac{d}{d+1}
\end{displaymath}
and
$\mathbb{P}(V_d) > 0 $ if
\begin{displaymath}
\mathbb{E} \left [ \left(\frac{d}{d+1} \right)^N \right] < \frac{d-1}{d}.
\end{displaymath}
\end{teo}

\begin{cor}
\label{C: MCC1HP}
Consider the process $\procpg$.

\begin{itemize}
\item[$(i)$] $\mathbb{P}(V_d) = 0$ if  
\begin{equation}\label{C:tfp1}
(\lambda^2d + \lambda d + \lambda +d + 1)p  \leq \lambda + d + 1. 
\end{equation}
 
\item[$(ii)$] $\mathbb{P}(V_d) > 0$ if 
\begin{equation}\label{C:tfp2}
(\lambda^2d - \lambda^2+ \lambda d - \lambda +d )p  > \lambda + d + 1.
\end{equation}

\end{itemize}
\end{cor}

\begin{cor}
\label{C: MCC1HPY}
Consider the process $\procyb$.

\begin{itemize}
\item[$(i)$] $\mathbb{P}(V_d) = 0 $ if
\begin{equation}\label{C:tfy1}
_2 F_1 \left(1,1; 2 +\frac{1}{\lambda}; \frac{d(1-p) +1}{d+1} \right)  \geq \frac{d(\lambda +1)}{d +1 - p}.
\end{equation}

\item[$(ii)$] $\mathbb{P}(V_d) > 0 $ if
\begin{equation}\label{C:tfy2}
_2 F_1 \left(1,1; 2 +\frac{1}{\lambda}; \frac{d(1-p) +1}{d+1} \right) < \frac{(d^2-1)(\lambda +1)}{d(d+1-p)}.
\end{equation}
\end{itemize}
\end{cor}

Observe that for the process $\procc,\hbox{ when } \mathcal{C} \in \{\mathcal{Y}(\lambda), \mathcal{P}(\lambda)\}
\hbox{ and } \mathcal{E} \in \{\mathcal{G}(p), \mathcal{B}(p)\},$ by a coupling argument one can see that $\mathbb{P}(V_d)$ is a non-decreasing function of $\lambda$ and also of $p$. Moreover, the function $\lambda_c(p)$, defined by
$$\lambda_c(p):=\inf\{\lambda: \mathbb{P}(V_d) >0 \},$$
is a non-increasing function of $p$, with $\lambda_c(1)=0$ and $\lambda_c(0)=\infty$.

\begin{defn}
Let $\eta_t$ be a $\procc \hbox{ for } \mathcal{C} \in \{\mathcal{Y}(\lambda),
\mathcal{P}(\lambda)\} \hbox{ and } \mathcal{E} \in \{\mathcal{G}(p), \mathcal{B}(p)\},$
with $0<p<1$. We say that $\eta_t$ exhibits \textit{phase transition} on $\lambda$ if
$0<\lambda_c(p)<\infty.$
\end{defn}

Machado \textit{et al.}(2016) proved phase transition on $\lambda $ for the process  $\procyb$. 
So, there exists a function
$\lambda_c(\cdot):(0,1)\rightarrow \mathbb{R}^+$ whose graphic
splits the parametric space $\lambda \times p$ into two regions. For those values 
of $(\lambda,p)$ above the curve $\lambda_c(p)$,  there is survival in $\procyb$ with positive probability.
Moreover, for those values of $(\lambda,p)$  below the curve $\lambda_c(p)$ extinction occurs in $\procyb$ 
with probability 1.

However,
it is not known anything about the continuity and strict monotonicity (in $p$) of the function $ \lambda_c(p) $. If there is continuity and strict monotonicity, then the process also has phase transition in $p$ for each $\lambda \in (0, \infty)$ fixed.

In order to answer the question about phase transition on $p$
for the process $\procc, \\ \hbox{ when } \mathcal{C} \in \{\mathcal{Y}(\lambda), \mathcal{P}(\lambda)\}
\hbox{ and } \mathcal{E} \in \{\mathcal{G}(p), \mathcal{B}(p)\},$
we start with the following definition
\[p_c(\lambda):=\inf\{p: \mathbb{P}(V_d) > 0 \}.\]

\begin{defn}
Let $\eta_t$ be a $\procc \hbox{ for } \mathcal{C} \in \{\mathcal{Y}(\lambda),
\mathcal{P}(\lambda)\} \hbox{ and } \mathcal{E} \in \{\mathcal{G}(p), \mathcal{B}(p)\},$
with $\lambda \in (0, \infty)$ fixed. We say that $\eta_t$ exhibits \textit{phase transition} on $p$ if
$0<p_c(\lambda)<1.$
\end{defn}

The item $(i)$ of Corollary~\ref{C: MCC1HPY} coincides with item $(iii)$ of Theorem 3.1
from Machado {\it et al.}~\cite{MRS2015}. The novelty of Corollary~\ref{C: MCC1HPY} is
its item $(ii)$ which provides a suficient condition for survival. Corollary~\ref{C: MCC1HPY}
guarantees phase transition in $p$ for $\procyb$ for $\lambda$ large enough, and gives
lower and upper bounds for $\lambda_c(p)$.

\begin{exa}
Consider $\yb4$. The equalities in (\ref{C:tfy1}) and (\ref{C:tfy2}) provide lower and upper bounds, respectively, for $\lambda_c(p)$. See Figure \ref{F: LimitanteTransicaoAnt}. These bounds guarantees phase transition in $p$ for $\lambda>\lambda_4^*.$ Where $\lambda_d^*$ is an upper bound for $\displaystyle\lim_{p\rightarrow 1^-}\lambda_c(p)$, where the former is the solution for
  \begin{displaymath}
 _2 F_1 \left(1,1; 2 +\frac{1}{\lambda}; \frac{1}{d+1} \right) = \frac{(d^2-1)(\lambda +1)}{d^2},
 \end{displaymath}
see Corollary \ref{C: MCC1HPY} $(ii)$. The following table shows computations for $\lambda_d^*$
for some values of $d$

  \begin{center}
  	\begin{tabular}{|l|c|c|c|c|c|c|}\hline
  		$d$ & 2 &3  & 4     &  5      & 6    & 10         \\ \hline
  		$\lambda_d^*$ & 0.4555826 & 0.1613016& 0.08212601 &   0.04961835   & 0.03315455   & 0.01110147     \\ \hline
  	\end{tabular}
  \end{center}

\begin{figure}[ht]
	
	\begin{tabular}{ccc}
		$\lambda$ & \parbox[c]{8cm}{\includegraphics[trim={1cm 1.5cm 1cm 1.5cm}, clip, width=8cm]{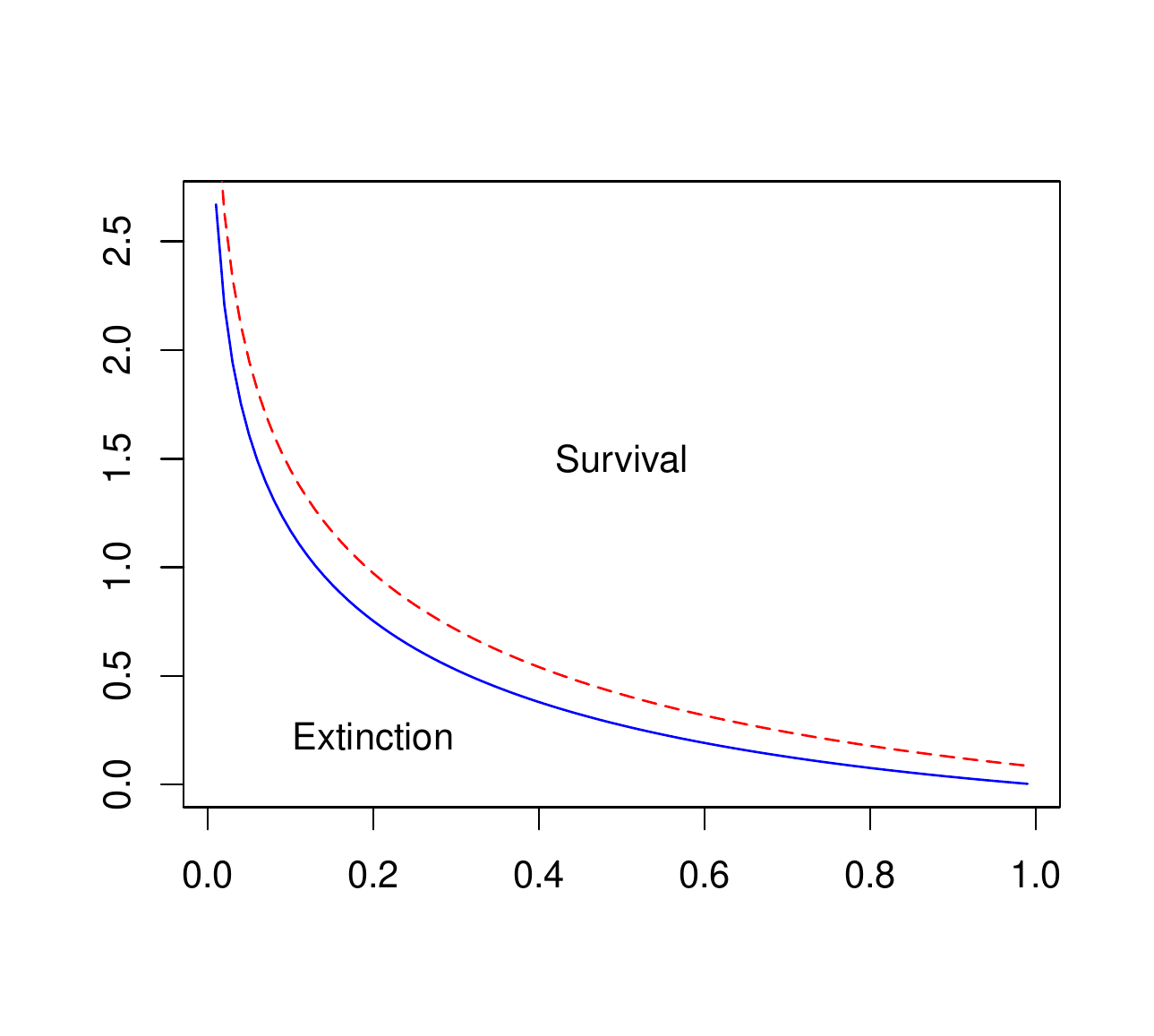}} & 
		\begin{tabular}{l}
			\textcolor{red}{- - - } Equality at (\ref{C:tfy1})\\ 
			\textcolor{blue}{------} Equality at (\ref{C:tfy2}) 
		\end{tabular} \\
		& $p$
	\end{tabular}
		\caption{Lower and upper bounds for $\lambda_c(p)$ in $\yb4$}
		\label{F: LimitanteTransicaoAnt}
\end{figure}

\end{exa}	

\begin{exa} Consider $\{\bbT^{4}; \mathcal{P}(\lambda), \mathcal{G}(p) \}$. The equalities in (\ref{C:tfp1}) and (\ref{C:tfp2}) provide lower and upper bounds, respectively, for $\lambda_c(p)$. See Figure \ref{F:Poison-Gem-d4}.  These bounds guarantees phase transition in $p$ for $\lambda>\lambda_4^*.$ Where $\lambda_d^*$ is an upper bound for $\displaystyle\lim_{p\rightarrow 1^-}\lambda_c(p)$, where the former is the solution for  
	\begin{displaymath}
	(\lambda^2d - \lambda^2+ \lambda d - \lambda +d )p  = \lambda + d + 1,
	\end{displaymath} 
	when $p=1,$ see Corollary \ref{C: MCC1HP} $(ii)$. Thus, $$\lambda_d^*=\frac{1}{d-1}.$$

\begin{figure}[ht]
	
\begin{tabular}{ccc}
	$\lambda$ & \parbox[c]{8cm}{\includegraphics[trim={1cm 1.5cm 1cm 2cm}, clip, width=8cm]{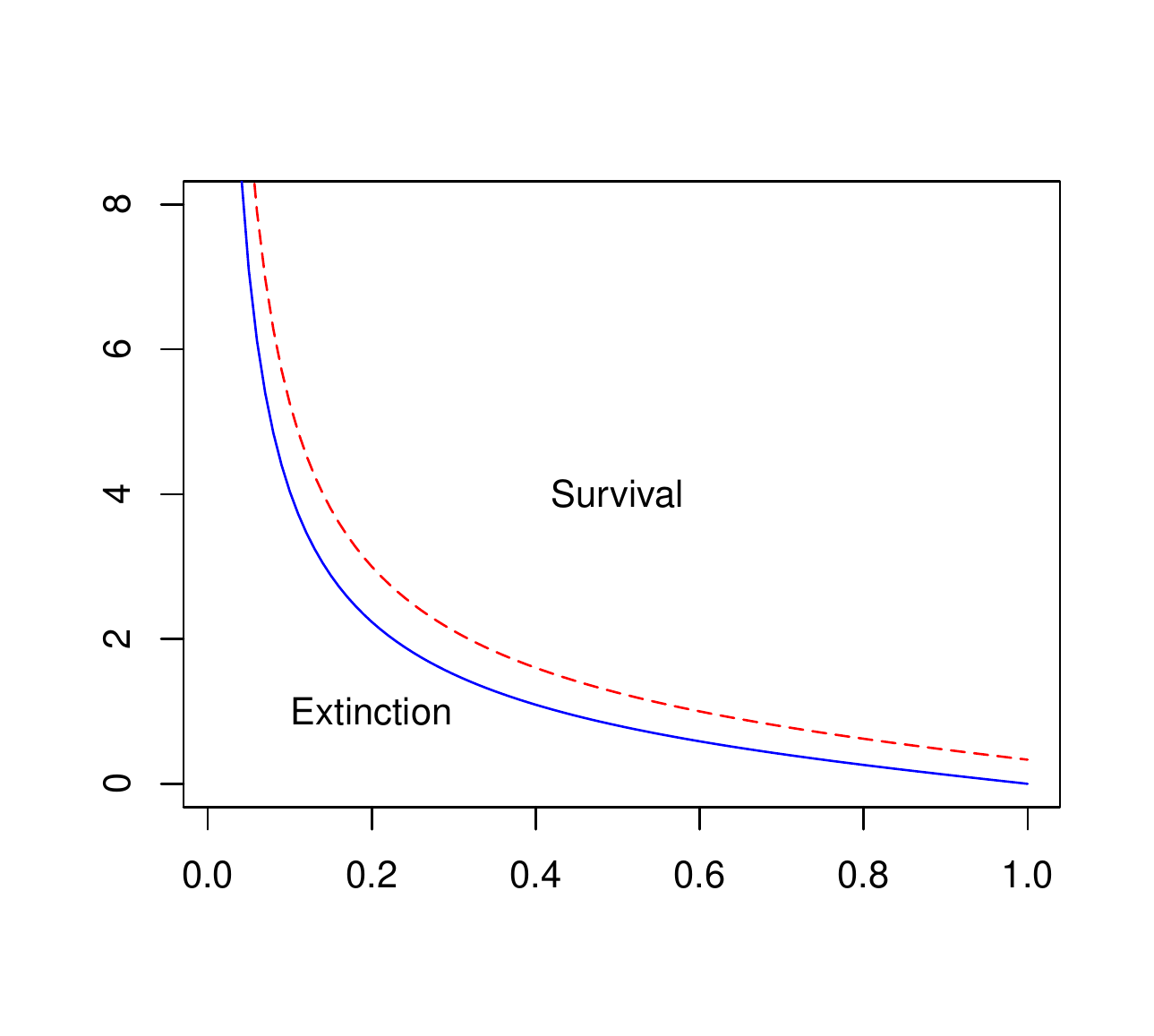}} & 
	\begin{tabular}{l}
		\textcolor{red}{- - - } Equality at (\ref{C:tfp1}) \\ 
		\textcolor{blue}{------} Equality at (\ref{C:tfp2})
	\end{tabular} \\
	& $p$
\end{tabular}
	\caption{Lower and upper bounds for $\lambda_c(p)$ in $\{\bbT^{4}; \mathcal{P}(\lambda), \mathcal{G}(p) \}$}
	\label{F:Poison-Gem-d4}
\end{figure}
	
\end{exa}	

\subsection{Probability of Survival}
\label{SS: SP}

We denote by $T(n,k)$ the number of surjective functions $f:A \to B $, where $|A| = n$ and $|B| = k$, whose value is given, by the inclusion-exclusion principle (see  Tucker~\cite{Tucker} p. 319), by
\begin{displaymath}
T(n,k) = \sum_{i=0}^{k} \left [ (-1)^i \binom{k}{i}(k-i)^n \right], n \geq k.
\end{displaymath}


\begin{teo}
\label{T: MCCHPE1}
Consider the process $\procT$. We have that
\begin{displaymath}
 \sum_{r=1}^{d+1} \left [(1 - \rho^r)\binom{d+1}{r}\sum_{n=r}^{\infty}\frac{T(n,r)}{(d+1)^n}\mathbb{P}(N=n) \right] \leq \mathbb{P}(V_d) \leq 1-\psi
\end{displaymath}
where $\psi$ and $\rho$ are, respectively, the smallest non-negative solutions of
\begin{align*}
&\sum_{y=0}^{d+1} \left[ s^y\binom{d+1}{y}\sum_{n=y}^{\infty}\frac{T(n,y)}{(d+1)^n}\mathbb{P}(N=n) \right] = s,\\
&\sum_{y=1}^{d} \left [ s^y\binom{d}{y}\sum_{n=y}^{\infty}\frac{T(n,y)+T(n,y+1)}{(d+1)^n}\mathbb{P}(N=n)\right ] = s - \sum_{n=0}^{\infty}\frac{\mathbb{P}(N=n)}{(d+1)^n} .
\end{align*}
\end{teo}

\begin{teo}
\label{T: MCCHPE1L}
Consider the process $\procT$. We have that
\[ \lim_{d \to \infty} \mathbb{P}(V_d) = 1 - \nu
\]
where  $\nu$ is the smallest non-negative solution of $ \mathbb{E}(s^N) = s.$
\end{teo}

\begin{cor}
\label{C: MCC1HPS}
Consider the process $\procpg$. Then
\[ \lim_{d \to \infty} \mathbb{P}(V_d) = \max \left \{ 0, \frac{p(\lambda^2 + \lambda +1)}{\lambda (1+ \lambda p)} \right \}.
\]
\end{cor}

\begin{exa}
Consider the process $\pgd$. If $d=10$ then
\begin{displaymath}
\mathbb{P}(N = n ) = \left\{
                      \begin{array}{ll}
                        \frac{9}{50}\left(\frac{5}{6} \right )^n, & \hbox{$n \geq 1$;} \\
                        \frac{1}{10}, & \hbox{$n=0$.}
                      \end{array}
                    \right.
\end{displaymath}
By using Theorem~\ref{T: MCCHPE1} we have that $\psi = 0.12226$ and $\rho = 0.143256$.
Then
\[ 0.8733 \leq \mathbb{P}(V_{10}) \leq 0.8778.\]
Besides
\[ \displaystyle \lim_{d \to \infty} \mathbb{P}(V_d) = 0.93. \]
\end{exa}

\begin{cor}
\label{C: MCC1HPS2}
Consider the process $\procyb$. Then
\[ \lim_{d \to \infty} \mathbb{P}(V_d) =1 - \nu
\]
where  $\nu$  is the smallest non-negative solution of
\[ _2 F_1 \left ( 1,1;2 + \frac{1}{\lambda}; p(s-1) + 1  \right ) = \frac{ s(\lambda +1)}{p(s + 1)}.
\]
\end{cor}
\begin{exa}
Consider the process $\procyb$. If $\lambda = 2$ and $p = 0.5$ then, by using
Corolary~\ref{C: MCC1HPS2}
\[ \displaystyle \lim_{d \to \infty} \mathbb{P}(V_d) = 0.680977.\]
\end{exa}

\subsection{The reach of the process}
\label{SS: RP}

In order to show results for the reach of the process, meaning the distance from the origin
to the furthest vertex where a colony is created, let us define a few technical quantities

\begin{defn}
\begin{align*}
\alpha&= d \left [ 1 -  \mathbb{E} \left [ \left(\frac{d}{d+1} \right)^N \right]  \right ] \\
\beta &= (d+1) \left [ 1 -  \mathbb{E} \left [ \left(\frac{d}{d+1} \right)^N \right]  \right ] = \alpha + 1- \mathbb{E} \left [ \left(\frac{d}{d+1} \right)^N \right]\\
D &= \max \left \{ 2; \frac{\beta}{\beta - \mathbb{P}(N \neq 0)} \right\} \\
B& =d(d-1)\left[ 1 - 2\mathbb{E} \left ( \left(\frac{d }{d+1} \right)^N \right) + \mathbb{E} \left ( \left(\frac{d-1 }{d+1} \right)^N \right) \right]
\end{align*}
\end{defn}

\begin{teo}\label{T: LEIT}
Consider the process $\procT$.  Assume that
\begin{displaymath}
\mathbb{E} \left [ \left(\frac{d}{d+1} \right)^N \right]  > \frac{d-1}{d}
\end{displaymath}
We have that
\begin{displaymath}
\frac{[1+D(1-\beta)][1-\beta^{m+1}]}{1+ D(1-\beta)-\beta^{m+1}} \leq \mathbb{P}(M_d \leq m) \leq \frac{[1 + \frac{\alpha(1-\alpha)}{B}](1-\alpha^{m+1})}{ 1 + \frac{\alpha(1-\alpha)}{B} - \alpha^{m+1}}
\end{displaymath}
and
\begin{displaymath}
 \frac{\alpha^2}{2(B+ \alpha)} + \alpha(1-\alpha)\frac{\ln \left[1 - \frac{\alpha B}{B + \alpha(1-\alpha)}\right]}{B \ln \alpha} \leq \mathbb{E}(M_d) \leq \frac{D \beta}{D+1} + D(1-\beta)\frac{\ln \left[1 - \frac{\beta}{1 + D(1-\beta)}\right]}{\ln \beta}.
\end{displaymath}
\end{teo}

\begin{cor} \label{C: CTM1}
Consider the process $\procpg$.  If
\[ (\lambda^2d + \lambda d + \lambda +d + 1)p  < \lambda + d + 1\]
then Theorem~\ref{T: LEIT} holds under the values
\begin{align*}
\alpha &= \frac{dp(\lambda + 1)^2}{(d + \lambda + 1)(\lambda p + 1)},\ \beta = \frac{(d+1)p(\lambda + 1)^2}{(d + \lambda + 1)(\lambda p + 1)},\
D = \max \left \{ 2; \frac{(d+1)(\lambda + 1)}{d \lambda} \right\},\\
B& = 2d(d-1) \left [ \frac{(\lambda +1)^2(2(\lambda p + 1) - 1) + (\lambda +1)(p-1)d - (\lambda p+1)d^2}{(d + 2\lambda  + 1)(d + 2 \lambda + 1)(\lambda p + 1)} \right ].
\end{align*}
\end{cor}

\begin{teo}\label{T: LEITL}
Consider the process $\procT$. We have that
\begin{displaymath}
M_d \overset{D}{\to} M,
\end{displaymath}
where $\mathbb{P}( M \leq m) = g_{m+1}(0)$, being $g(s) = \mathbb{E}(s^N) $ and $g_{m+1}(s) = \overset{ m+1 \textrm { times }} {g(g(\cdots g(s)) \cdots )}$.
\end{teo}

\begin{cor}
\label{C: LEITL}
Consider the process $\procpg$.
\begin{itemize}
\item[$(i)$] If
$ p \neq {(\lambda^2 + \lambda + 1)}^{-1}$ then 
\[ \mathbb{P}(M \leq m) = \frac{1 - \left (\frac{(\lambda + 1)^2 p}{\lambda p + 1} \right )^{m+1}}{1 - \frac{\lambda(\lambda p +1)}{(1-p)(\lambda p + 1)}\left (\frac{(\lambda +1)^2 p}{\lambda p + 1} \right )^{m+1}}, \ m \geq 0
\] and
\[ \mathbb{E}(M)= \frac{(1- p(\lambda^2 + \lambda + 1))}{\lambda (\lambda p + 1))} \lim_{ s \to \infty} \left [ \psi_{\gamma} \left( 1 - \frac{\ln{\frac{(\lambda +1)(1-p)}{\lambda ( \lambda p + 1)}}}{ \ln {\gamma}} \right) - \psi_{\gamma} \left( s - \frac{\ln{\frac{(\lambda +1)(1-p)}{\lambda ( \lambda p + 1)}}}{ \ln {\gamma}} + 1 \right) \right ]  \]
where $\displaystyle \gamma = \frac{(\lambda +1)^2 p}{\lambda p + 1}$ and  
$\displaystyle \psi_a(z) = -\ln(1-a) + \ln(a) \sum_{n=0}^{\infty} \frac{a^{n+z}}{1-a^{n+z}},$
being $\psi_a(z)$ known as the $a$-digama function.\\ 
\item[$(ii)$] If
$ p = {(\lambda^2 + \lambda + 1)}^{-1}$ then
\[ \mathbb{P}(M \leq m) = \frac{(m+1)\lambda}{(m+1)\lambda +1}\]
and
\[ \mathbb{E}(M)= \infty.\]
\end{itemize}
\end{cor}

\subsection{Number of collonies in the process}
\label{SS: NC}

\begin{teo}
\label{T: NC}
Consider the process $\procT$. If
\begin{displaymath}
\mathbb{E} \left [ \left(\frac{d}{d+1} \right)^N \right]  > \frac{d-1}{d}
\end{displaymath}
then
\begin{displaymath}
\mathbb{E}(I_d) \leq (1 - \beta)^{-1} \textrm { and }
\end{displaymath}
\begin{displaymath}
\mathbb{E}(I_d) \geq \sum_{r=1}^{d+1} \left [[1 + r\theta] \dbinom{d+1}{r} \sum_{n=r}^{\infty}\frac{T(n,r)}{(d+1)^n}\mathbb{P}(N=n) \right] + \mathbb{P}(N=0)
\end{displaymath}
\noindent
where $\theta = (1 - \alpha)^{-1}.$
Besides that, if $\mathbb{E}(N) < 1$ (the subcritical case)
\[
\lim_{d \to \infty} \mathbb{E}(I_d) = \frac{1}{1 - \mathbb{E}(N)}. \]
\end{teo}

\section{Proofs}
\label{S: Proofs}

In order to prove the main results we define auxiliary processes whose understanding will provide
bounds for the processes defined at introduction. In the first two auxiliary process, denoted by $\up$ and $\upP$,
every time a colony collapses the survival individuals are only allowed to choose neighbour
vertices which are further (compared to the origin) that the vertex where their colony was placed.
In other words an individual is not allowed to choose the neighbour vertex which has been already colonized. We refer to this process as \textit{Self Avoiding}.
The last two auxiliary process, denoted by $\low$ and $\lowP$, while the survival individuals are allowed to choose
the neighbour vertex which has been already colonized, those who does that are not able to colonize it
as this place is considered hostile or infertile. We refer to this process as \textit{Move Forward or Die}.
In both processes, $Y$,  the number of new colonies at collapse times in a vertex $x$ equals the number of diferent neighbours chosen which are located further from the origin than $x$ is. Besides that, every new colony starts with only one individual.

\begin{prop}
\label{P: CFGP}
Consider a sequence of random variables $\{Y_d\}_{ d \in \mathbb{N}}$ whose sequence of probability
generating functions is $\{g_{Y_d}(s)\}_{ d \in \mathbb{N}}$ and a random variable $Y$ such that $Y_d \overset{D}{\to} Y$. Then $g_{Y_d, m}(s)$, the $m-th$ composition of $g_{Y_d}(s)$,  converges to
$ g_{Y,m}(s)$, where $ g_{Y,m}(s)$ is the $m-th$ composition of $g_Y(s)$, the probability
generating function of $Y$.
\end{prop}
\begin{proof}[Proof of Proposition~\ref{P: CFGP}]
From the fact that  $Y_d \overset{D}{\to} Y$ it follows that  
$\displaystyle g_Y(s) = \lim_{d \to \infty} g_{Y_d}(s)$.
 \[ \lim_{d \to \infty} g_{Y_d, 2}(s) = \lim_{d \to \infty} g_{Y_d}(g_{Y_d}(s)) = \lim_{d \to \infty} \mathbb{E} \left [ \left ( \mathbb{E}(s^{Y_d}) \right)^{Y_d}  \right ]  \]

From the Dominated Convergence Theorem~\cite[Theorem 9.1 page 26]{Thorisson} (observe that $[\mathbb{E}(s^{Y_d})]^{Y_d} \in [0,1]$)
 \[ \lim_{d \to \infty} \mathbb{E} \left [ [\mathbb{E}(s^{Y_d})]^{Y_d} \right ] =  \mathbb{E} \left [ \lim_{d \to \infty} [\mathbb{E}(s^{Y_d})]^{Y_d} \right].
 \]

Again, from the Dominated Convergence Theorem~\cite[Theorem 9.1 page 26]{Thorisson} 
(observe that $s^{Y_d} \in [0,1]$ and that $Y_d$ converges to $Y$ in distribution) $Y_d  \ln \mathbb{E}(s^{Y_d})$ converges in distribution to $Y \ln \mathbb{E}(s^Y).$
So we conclude that
\[ \mathbb{E} \left [ \lim_{d \to \infty}  \left ( \mathbb{E}(s^{Y_d}) \right)^{Y_d}  \right ] = \mathbb{E} \left [ [ \mathbb{E}(s^Y) ]^Y  \right ]
 \]
and then
 \[ \lim_{d \to \infty} g_{Y_d, 2}(s)= \lim_{d \to \infty} \mathbb{E} \left [ \left ( \mathbb{E}(s^{Y_d}) \right)^{Y_d}  \right ]  = \mathbb{E} \left [ [ \mathbb{E}(s^Y) ]^Y  \right ] = g_{Y,2}(s).
 \]

By induction one can prove that $\displaystyle\lim_{d \to \infty} g_{Y_d, m}(s) = g_{Y,m}(s).$
\end{proof}

\begin{prop}
\label{P: ConvBranching}
Let $\{Z_n \}_{n \geq 0}, \{Z_{1,n} \}_{n \geq 0}, \{Z_{2,n} \}_{n \geq 0}, \cdots$ be a branching
processes and $Y, Y_1, Y_2, \cdots $, respectively, their offspring distributions. Supose that
\begin{enumerate}
\item[(i)] $Y_d \overset{D}{\to} Y$;
\item[(ii)] $\mathbb{P}(Y_d \geq k) \leq \mathbb{P}(Y_{d+1} \geq k)$, for all $k$ and for all $d$;
\end{enumerate}
Then, if $\nu_d$ is the pro\-ba\-bi\-li\-ty of the extinction of the process $\{Z_{d,n} \}_{n \geq 0}$  and $\nu$ is the pro\-ba\-bi\-li\-ty of the extinction of the process $\{Z_{n} \}_{n \geq 0}$ we have that
\[ \lim_{d \to \infty}\nu_d = \nu.
\]
\end{prop}

\begin{proof}[Proof of Proposition~\ref{P: ConvBranching}]
From \textit{(i)}, \textit{(ii)} and by using a coupling argument we have that

\begin{equation}\label{1}
 \nu_d \ge \nu_{d+1} \ge \lim_{d \to \infty} \nu_d =: \nu_L  \ge \nu.
\end{equation}

From the fact that $Y_d \overset{D}{\to} Y$ and \cite[Theorem 25.8, page 335]{Billingsley} we have that 

\begin{equation}\label{A}
\phi_d(s):=\mathbb{E}[s^{Y_d}]\underset{d\to\infty}\longrightarrow\mathbb{E}[s^Y] := \phi(s).
\end{equation}

Let $s\in[0,1]$ fixed and $f(y):=s^y, \ y\in\mathbb{N}$. Clearly, $f$ is non-increasing and therefore from \textit{(ii)} and \cite[equation (3.3), page 6]{Thorisson} we have that 
\begin{equation}\label{B}
\phi_{d+1}(s)\leq \phi_{d}(s).
\end{equation}

\noindent
From (\ref{A}), (\ref{B}) and Dini's Theorem, we have that 
\begin{equation}\label{C}
\phi_{d} (\cdot) \longrightarrow \phi(\cdot)  \text{ uniformly.}
\end{equation}

\noindent
From (\ref{1}), (\ref{C}) and \cite[Exercise 9 - Chapter 7]{Rudin}:
\begin{equation}\label{D}
\lim_{d\to \infty}\phi_d(\nu_d)=\phi(\nu_L)
\end{equation}

\noindent
Finally, given that $\phi_d(\nu_d)=\nu_d$, from (\ref{D}) we obtain that 
\begin{equation}\label{E}
\phi(\nu_L)=\nu_L.
\end{equation}

\noindent
From the convexity of $\phi(s)$ it follows that $\phi(s) = s$ (the fixed points of $\phi(\cdot)$) for at most two points in $[0,1]$. It is known that [see~\cite{Harris2002}, Theorem 6.1 and its proof] if $\nu<1$, the fixed points of $\phi(\cdot)$ are  $s=\nu$ and $s=1$. If $\nu=1$, the unique solution is 1. 
So there are two cases to be considered.

\textbf{1.} If $\nu_d < 1$ for some $d\geq 1$, then from (\ref{1}) it follows that $\nu_L < 1$. If $\nu_L < 1$, it follows from (\ref{E}) that $\nu_L = \nu$.

\textbf{2.}  If $\nu_d = 1$ for all $d\geq1$, then
\[ \mathbb{E}(Y_d) \leq 1 \textrm { for all } d\geq1. \]
Then,
\begin{equation} \label{lim}
	\lim_{d \to \infty} \mathbb{E}(Y_d) \leq 1.
\end{equation}

From \textit{(ii)} we have that $\mathbb{P}(Y_d \geq k) \leq \mathbb{P}(Y \geq k)$, for all $k$ and all $d$. From \textit{(i)}, \textit{(ii)} and a non standart version of Fatou Lemma~\cite[page 230]{Ash} (applied to the sequence $a_{d,j} = j \bbP(Y_d=j)$), it follows that

\begin{equation} \label{TCD}
\liminf_{d \to \infty} \mathbb{E}(Y_d) \ge \mathbb{E}(Y).
\end{equation}

From (\ref{lim}) and (\ref{TCD}), it follows that $\nu=1.$ Then, from (\ref{E}) we have that $\nu_L=\nu.$
\end{proof}

\subsection{$\up$: The Self Avoiding model}
\hfill

\begin{prop}
\label{P: CCSRSR3}
Consider the process $\up$. $\mathbb{P}(V_d) > 0 $ if and only if
\begin{displaymath}
\mathbb{E} \left [ \left(\frac{d - 1}{d} \right)^N \right]  < \frac
{d-1}{d}
\end{displaymath}
\end{prop}

\begin{proof}[Proof of Proposition~\ref{P: CCSRSR3}] First of all observe that for a fixed distribution for $N$, the processes $\up$ and $\upP$ either both survives or both die. Next observe that the process
$\upP$ behaves as a homogeneous branching process. Every vertex $x$ which is colonized produces
$Y_d$ new colonies (whose distribution depends only on $N$) on the $d$ neighbour vertices which
located are further from the origin than $x$ is.
By conditioning one can see that
\begin{equation}
\label{E: MediaY}
 \mathbb{E}(Y_d) = d \sum_{n=0}^{\infty} \left [\left(1 - \left(\frac{d-1}{d}\right)^n\right)\mathbb
{P}(N = n) \right] =  d  \left[ 1-  \mathbb{E}\left [\left(\frac{d-1}{d} \right)^N \right] \right].
\end{equation}

From the theory of homogeneous branching processes we see that $\upP$ (and also $\up$) survives if and only if $\mathbb{E}[ (\frac{d - 1}{d} )^N ]  < \frac
{d-1}{d}.$
\end{proof}

\begin{prop}
\label{P: CCSRSR1V}
Consider the process $\up$. Then
\begin{displaymath}
\mathbb{P}(V_d) = \sum_{r=1}^{d+1}\left [(1 - \psi^r)\binom{d+1}
{r}\sum_{n=r}^{\infty}\frac{T(n,r)}{(d+1)^n}\mathbb{P}(N=n) \right]
\end{displaymath}
where $\psi$, the extinction probability for the process $\upP$, is the smallest non-negative solution of
\begin{displaymath}
\sum_{y=0}^{d} \left[ s^y\binom{d}{y}\sum_{n=y}^{\infty}\frac{T(n,y)}
{d^n}\mathbb{P}(N=n) \right] = s.
\end{displaymath}
On the sub critical regime, which means
\begin{displaymath}
\mathbb{E} \left [ \left(\frac{d - 1}{d} \right)^N \right]  > \frac
{d-1}{d},
\end{displaymath}
it holds that
\begin{displaymath}
\mathbb{E}(I_d) = \sum_{r=1}^{d} \left [[1 + r\theta_u]\dbinom{d+1}{r} \sum_
{n=r}^{\infty}\frac{T(n,r)}{(d+1)^n}\mathbb{P}(N=n) \right] + \mathbb{P}(N=0)
\end{displaymath}
where
\begin{displaymath}
\theta_u = \left \{ 1- d  \left [ 1 - \mathbb{E} \left ( \left(\frac{d -
1}{d} \right)^N \right) \right ]\right\}^{-1}.
\end{displaymath}
\end{prop}
\begin{proof}[Proof of Proposition~\ref{P: CCSRSR1V}]
Let $Y_{d,R}$ be the number of colonies created at the neighbour vertices of the origin from
its colony at the collapse time. Then
\begin{displaymath}
\mathbb{P}(V_d)  = \sum_{r=0}^{d+1} \mathbb{P}(V_d| Y_{d,R} = r) \mathbb{P}
(Y_{d,R} = r)
\end{displaymath}
where
\begin{displaymath}
\mathbb{P}(Y_{d,R}=r) =\sum_{n=r}^{\infty} \left[ \mathbb{P}(N=n) \frac
{\dbinom{d+1}{r}T(n,r)}{(d+1)^n} \right] \textrm { for } r =0,1,2,
\cdots, d+1.
\end{displaymath}
because
\begin{displaymath}
\mathbb{P}(Y_{d,R}=r | N = n) =\frac
{\dbinom{d+1}{r}T(n,r)}{(d+1)^n}.
\end{displaymath}

Given that $Y_{d,R} = r$ one have $r$ independent $\upP$ processes living on $r$ independent
rooted trees. Every vertex $x$ which is colonized, on some of these trees, right after the collapse will
have $N$ survival individuals. These individuals will produce
$Y_d$ new colonies (whose distribution depends only on $N$) on the $d$ neighbour vertices which are
located further from the origin than $x$ is. So we have that
\begin{displaymath}
\mathbb{P}(Y_d=y | N = n) = \frac{\dbinom{d}{y}T(n,y)}{d^n}.
\end{displaymath}
From this,
\begin{displaymath}
\mathbb{P}(Y_d=y) =\sum_{n=y}^{\infty} \left[ \mathbb{P}(N=n) \frac
{\dbinom{d}{y}T(n,y)}{d^n} \right] \textrm { for } y =0,1,2, \cdots,
d,
\end{displaymath}
and
\begin{displaymath}
\mathbb{E}(s^{Y_d})  = \sum_{y=0}^{d} \left[ s^y\binom{d}{y}\sum_{n=y}^
{\infty}\frac{T(n,y)}{d^n}\mathbb{P}(N=n) \right].
\end{displaymath}
Then $\mathbb{P}(V_d^C | Y_{d,R} = r) = \psi^r \hbox{ for } r =0,1,2,
\cdots, d+1$ and
\begin{displaymath}
\mathbb{P}(V_d) = \sum_{r=1}^{d+1}\left [(1 - \psi^r)\binom{d+1}
{r}\sum_{n=r}^{\infty}\frac{T(n,r)}{(d+1)^n}\mathbb{P}(N=n) \right]
\end{displaymath}
As for the second part of the proposition
\begin{displaymath}
\mathbb{E}(I_d)  = \sum_{r=0}^{d+1} \mathbb{E}(I_d| Y_{d,R} = r) \mathbb{P}
(Y_{d,R}= r).
\end{displaymath}
Besides that, $\mathbb{E}(I_d| Y_{d,R} = r) = r\theta_u + 1$ 
(see Stirzaker~\cite[Exercise 2b, page 280]{Stirzaker}).
\end{proof}

\begin{prop}
\label{P: CCSRSR1VL}
Consider the process $\up$. Then
\begin{equation}
\label{limPVD}
\lim_{d \to \infty} \mathbb{P}(V_d) = 1 - \nu
\end{equation}
where  $\nu$ is the smallest non-negative solution of $ \mathbb{E}(s^N) = s$. Besides that, if 
$\mathbb{E}(N) < 1$ (the subcritical case) then
\begin{equation}
\label{limEID}
\lim_{d \to \infty} \mathbb{E}(I_d) = \frac{1}{1 - \mathbb{E}(N)}.
\end{equation}
\end{prop}

\begin{proof}[Proof of Proposition~\ref{P: CCSRSR1VL}]
In order to prove~(\ref{limPVD}) one has to apply Proposition~\ref{P: ConvBranching}, observing that $Y_d \overset{D}{\to} N$ and $Y_{d,R} \overset{D}{\to} N.$
Moreover to prove~(\ref{limEID}) observe that
\[
\lim_{d \to \infty} \mathbb{E}(I_d) = \lim_{d \to \infty} \sum_{r=0}^{d+1} \mathbb{E}(I_d|Y_{d,R}=r)\bbP(Y_{d,R}=r).\]
As  $Y_d \overset{D}{\to} N$ and $Y_{d,R} \overset{D}{\to} N $ then
\[ \lim_{d \to \infty}  \mathbb{E}(I_d|Y_{d,R}=r) = \lim_{d \to \infty} r\theta_u +1 = \frac{r}{1-\mathbb{E}(N)}+1 \]
and the result follows from the Dominated Convergence Theorem~\cite[Theorem 9.1 page 26]{Thorisson}.
\end{proof}

\begin{prop} \label{P: LIMTSRSR1}
Consider the process $\upP$.
Assuming
\begin{displaymath}
\mathbb{E} \left [ \left(\frac{d - 1}{d} \right)^N \right]  > \frac
{d-1}{d}
\end{displaymath}
we have that
\begin{displaymath}
\frac{[1+D(1-\mu)][1-\mu^{m+1}]}{1+ D(1-\mu)-\mu^{m+1}} \leq \mathbb{P}(M_d \leq m) \leq \frac{[1 + \frac{\mu(1-\mu)}{B}](1-\mu^{m+1})}{ 1 + \frac{\mu(1-\mu)}{B} - \mu^{m+1}}
\end{displaymath}
and
\begin{displaymath}
 \frac{\mu^2}{2(B+ \mu)} + \mu(1-\mu)\frac{\ln \left[1 - \frac{\mu B}{B + \mu(1-\mu)}\right]}{B \ln \mu} \leq \mathbb{E}(M_d) \leq \frac{D \mu}{D+1} + D(1-\mu)\frac{\ln \left[1 - \frac{\mu}{1 + D(1-\mu)}\right]}{\ln \mu}
\end{displaymath}
where
\begin{align*}
\mu &= d \left [ 1 -  \mathbb{E} \left [ \left(\frac{d}{d+1} \right)^N \right]  \right ] \\
D &= \max \left \{ 2; \frac{g^{\prime}(1)}{g^{\prime}(1) - \mathbb{P}(N \neq 0)} \right\} \\
B & = d(d-1)\left[ 1 - 2\mathbb{E} \left ( \left(\frac{d-1 }{d} \right)^N \right) + \mathbb{E} \left ( \left(\frac{d-2 }{d} \right)^N \right) \right].
\end{align*}
Moreover,
\begin{displaymath}
M_d \overset{D}{\to} M,
\end{displaymath}
where $\mathbb{P}( M \leq m) = g_{m+1}(0)$, being $g(s) = \mathbb{E}(s^N) $ and $g_{m+1}(s) = \overset{ m+1 \textrm { times }} {g(g(\cdots g(s)) \cdots )}$.
\end{prop}
\begin{proof}[Proof of Proposition~\ref{P: LIMTSRSR1}]

Every vertex $x$ which is colonized produces
$Y_d$ new colonies (whose distribution depends only on $N$) on the $d$ neighbour vertices which
are located further from the origin than $x$ is. The random variable $Y_d$ can be seen as $Y_d = \sum_{i=1}^{d}I_i$ where for $i=1, \dots, d$
\begin{displaymath}
I_i = \left\{%
\begin{array}{ll}
1, & \hbox{the $i-th$ neighbour of $x$ is colonized} \\
0, & \hbox{else.} \\
\end{array}%
\right.
\end{displaymath}
Defining $g_{Y_d}(s)$ as the generating function of $Y_d$ observe that equation~(\ref{E: MediaY}) gives $g_{Y_d}^{\prime}(1)$. Moreover
\begin{displaymath}
{Y_d}^2 = \left (\sum_{i=1}^{d}I_i \right )^2 = \sum_{i=1}^{d}I_i^2 + 2\sum_{1 \leq i < j \leq d}I_iI_j
\end{displaymath}
and
\begin{displaymath}
 \mathbb{E} \left({Y_d}^2 \right ) =  d \mathbb{E} \left (I_1^2 \right ) + d(d-1) \mathbb{E}(I_1I_2)
\end{displaymath}
and finally
\begin{displaymath}
 \mathbb{E} \left({Y_d}^2 \right ) =  d \left[ 1 - \mathbb{E} \left [\left( \frac{d-1}{d} \right)^N \right] \right]   + d(d-1) \left [  1 -2  \mathbb{E} \left [\left ( \frac{d-1}{d} \right )^N \right] + \mathbb{E} \left [\left ( \frac{d-2}{d} \right )^N \right] \right ].
\end{displaymath}
Then
\begin{align*}
g_{Y_d}^{\prime \prime}(1) = \mathbb{E} \left(Y_d(Y_d-1) \right ) =  d(d-1) \left [  1 -2  \mathbb{E} \left [\left ( \frac{d-1}{d} \right )^N \right] + \mathbb{E} \left [\left ( \frac{d-2}{d} \right )^N \right] \right ].
\end{align*}
Then the result follows from Theorem 1 page 331 in~\cite{AA}, where $ m = g_{Y_d}^{\prime}(1)$. 

The convergence $M_d \overset{D}{\to} M$ follows from the fact that $Y_d \overset{D}{\to} N$ when $ d \to \infty$ and from Proposition \ref{P: CFGP}.
\end{proof}

\subsection{{$\low$: Move Forward or Die}}
\hfill

\begin{prop}
\label{P: CCSRCR3}
Consider the process $\low$. $\mathbb{P}(V_d) > 0 $ if and only if
\begin{displaymath}
\mathbb{E} \left [ \left(\frac{d}{d+1} \right)^N \right]  < \frac{d-
1}{d}
\end{displaymath}
\end{prop}

\begin{proof}[Proof of Proposition~\ref{P: CCSRCR3}]
First of all observe that for a fixed distribution for $N$, the processes $\low$ and $\lowP$ either both survives or both die. Next observe that the process
$\lowP$ behaves as a homogeneous branching process. Every vertex $x$ which is colonized produces
a bunch of survival individuals right after the collapse which are willing to jump to one of the $d+1$
nearest neighbours vertices of $x$. All those which jump towards the origin get killed. So,
$Y_d$ new colonies will be found  on the $d$ neighbour vertices which are located further
from the origin than $x$ is.
By conditioning one can see that
\begin{equation}
\label{E: MediaYL}
 \mathbb{E}(Y_d) = d\sum_{n=0}^{\infty} \left [\left(1 - \left(\frac{d}{d+1}\right)^n\right)\mathbb
{P}(N = n) \right] =  d  \left[ 1-  \mathbb{E} \left [\left(\frac{d}{d+1} \right)^N \right ]\right]
\end{equation}

From the theory of homogeneous branching processes we see that $\lowP$ (and also $\low$) survives if and only if $\mathbb{E} \left [ \left(\frac{d}{d+1} \right)^N \right]  < \frac{d-1}{d}.$
\end{proof}

\begin{prop} \label{P: CCSRCR1V}
Consider the process $\low$. Then
\begin{displaymath}
\mathbb{P}(V_d) = \sum_{r=1}^{d+1}\left [(1 - \rho^r)\binom{d+1}
{r}\sum_{n=r}^{\infty}\frac{T(n,r)}{(d+1)^n}\mathbb{P}(N=n) \right]
\end{displaymath}
where $\rho$, the extinction probability for the process $\lowP$, is the smallest non-negative solution of
\begin{displaymath}
\sum_{y=0}^{d} \left[ s^y\binom{d}{y}\sum_{n=y}^{\infty}\frac{T(n,y)+
T(n,y+1)}{(d+1)^n}\mathbb{P}(N=n) \right] = s.
\end{displaymath}
On the subcritical regime, which means
\begin{displaymath}
\mathbb{E} \left [ \left(\frac{d}{d+1} \right)^N \right]  > \frac{d-
1}{d},
\end{displaymath}
it holds that
\begin{displaymath}
\mathbb{E}(I_d) = \sum_{r=1}^{d+1} \left [[1 + r\theta_l] \dbinom{d+1}{r}
\sum_{n=r}^{\infty}\frac{T(n,r)}{(d+1)^n}\mathbb{P}(N=n) \right] + \mathbb{P}(N=0)
\end{displaymath}
where
\begin{displaymath}
\theta_l = \left \{ 1- d  \left [ 1 - \mathbb{E} \left ( \left(\frac{d }{d
+1} \right)^N \right) \right ]\right\}^{-1}.
\end{displaymath}
\end{prop}

\begin{proof}[Proof of Proposition~\ref{P: CCSRCR1V}]
Let $Y_{d,R}$ be the number of colonies created at the neighbour vertices of the origin from
its colony at the collapse time. Then
\begin{displaymath}
\mathbb{P}(V_d)  = \sum_{r=0}^{d+1} \mathbb{P}(V_d| Y_{d,R} = r) \mathbb{P}
(Y_{d,R} = r)
\end{displaymath}
where
\begin{displaymath}
\mathbb{P}(Y_{d,R}=r) =\sum_{n=r}^{\infty} \left[ \mathbb{P}(N=n) \frac
{\dbinom{d+1}{r}T(n,r)}{(d+1)^n} \right] \textrm { for } r =0,1,2,
\cdots, d+1.
\end{displaymath}
because
\begin{displaymath}
\mathbb{P}(Y_{d,R}=r | N = n) = \frac
{\dbinom{d+1}{r}T(n,r)}{(d+1)^n}.
\end{displaymath}

Given that $Y_{d,R} = r$ one have $r$ independent $\lowP$ processes living on $r$ independent
rooted trees. Every vertex $x$ which is colonized, on some of these trees, right after the collapse will
have $N$ survival individuals. These individuals will produce
$Y_d$ new colonies (whose distribution depends only on $N$) on the $d$ neighbour vertices which are
located further from the origin than $x$ is. So we have that

\begin{displaymath}
\mathbb{P}(Y_d=y | N = n) = \frac{\dbinom{d}{y}[T(n,y)+ T(n,y+1)]}{(d
+1)^n}
\end{displaymath}
From this,
\begin{displaymath}
\mathbb{P}(Y_d=y) =\sum_{n=y}^{\infty} \left[ \mathbb{P}(N=n) \frac
{\dbinom{d}{y}[T(n,y)+ T(n,y+1)]}{(d+1)^n} \right] \textrm { for } y
=0,1,2, \cdots, d.
\end{displaymath}
and
\begin{displaymath}
\mathbb{E}(s^{Y_d})  = \sum_{y=0}^{d}s^y \sum_{n=y}^{\infty} \left[
\mathbb{P}(N=n) \frac{\dbinom{d}{y}[T(n,y)+ T(n,y+1)]}{(d+1)^n}
\right].
\end{displaymath}
Then $\mathbb{P}({V_d}^C | Y_{d,R} = r) = \rho^r,\ r =0,1,2,
\cdots, d+1$ and
\begin{displaymath}
\mathbb{P}(V_d) = \sum_{r=1}^{d+1}\left [(1 - \rho^r)\binom{d+1}
{r}\sum_{n=r}^{\infty}\frac{T(n,r)}{(d+1)^n}\mathbb{P}(N=n) \right]
\end{displaymath}
As for the second part of the proposition
\begin{displaymath}
\mathbb{E}(I_d)  = \sum_{r=0}^{d+1} \mathbb{E}(I_d| Y_{d,R} = r) \mathbb{P}
(Y_{d,R} = r).
\end{displaymath}
Besides that,  $\mathbb{E}(I_d| Y_{d,R} = r) = r\theta_l + 1$
(see Stirzaker~\cite[Exercise 2b, page 280]{Stirzaker}).
\end{proof}

\begin{prop} 
\label{P: RCR1VL2}
Consider the process $\low$. Then, 
\begin{equation}
\label{limPVD2}
\lim_{d \to \infty} \mathbb{P}(V_d) = 1 - \nu
\end{equation}
where  $\nu$ is the smallest non-negative solution of $ \mathbb{E}(s^N) = s$. Besides that, if 
$\mathbb{E}(N) < 1$ (the subcritical case) then
\begin{equation}
\label{limEID2}
\lim_{d \to \infty} \mathbb{E}(I_d) = \frac{1}{1 - \mathbb{E}(N)}. 
\end{equation}
\end{prop}

\begin{proof}[Proof of Proposition~\ref{P: RCR1VL2}]
In order to prove~(\ref{limPVD2}) one has to aply Proposition~\ref{P: ConvBranching}, observing that $Y_d \overset{D}{\to} N$ and $Y_{d,R} \overset{D}{\to} N.$
For the proof of~(\ref{limEID2}) observe that
\[
\lim_{d \to \infty} \mathbb{E}(I_d) = \lim_{d \to \infty} \sum_{r=0}^{d+1} \mathbb{E}(I_d|Y_{d,R}=r)\bbP(Y_{d,R}=r).\]
As  $Y_d \overset{D}{\to} N$ and $Y_{d,R} \overset{D}{\to} N $ then
\[ \lim_{d \to \infty}  \mathbb{E}(I_d|Y_{d,R}=r) = \lim_{d \to \infty} r \theta_l +1 = \frac{r}{1-\mathbb{E}(N)}+1. \]
The result follows from the Dominated Convergence Theorem~\cite[Theorem 9.1 page 26]{Thorisson}.
\end{proof}

\begin{prop} \label{P: LIMTSRCR1}
Consider the process $\lowP$.
Assuming
\begin{displaymath}
\mathbb{E} \left [ \left(\frac{d}{d+1} \right)^N \right]  > \frac{d-
1}{d}
\end{displaymath}
We have that
\begin{displaymath}
\frac{[1+D(1-\mu)][1-\mu^{m+1}]}{1+ D(1-\mu)-\mu^{m+1}} \leq \mathbb{P}(M_d \leq m) \leq \frac{[1 + \frac{\mu(1-\mu)}{B}](1-\mu^{m+1})}{ 1 + \frac{\mu(1-\mu)}{B} - \mu^{m+1}}
\end{displaymath}
and
\begin{displaymath}
 \frac{\mu^2}{2(B+ \mu)} + \mu(1-\mu)\frac{\ln \left[1 - \frac{\mu B}{B + \mu(1-\mu)}\right]}{B \ln \mu} \leq \mathbb{E}(M_d) \leq \frac{D \mu}{D+1} + D(1-\mu)\frac{\ln \left[1 - \frac{\mu}{1 + D(1-\mu)}\right]}{\ln \mu}
\end{displaymath}
where
\begin{align*}
\mu &= d \left [ 1 -  \mathbb{E} \left [ \left(\frac{d}{d+1} \right)^N \right]  \right ] \\
D &= \max \left \{2; \frac{\mu}{\mu - \mathbb{P}(N \neq 0)} \right\} \\
B& = d(d-1)\left[ 1 - 2\mathbb{E} \left ( \left(\frac{d }{d+1} \right)^N \right) + \mathbb{E} \left ( \left(\frac{d-1 }{d+1} \right)^N \right) \right].
\end{align*}
Besides that,
\begin{displaymath}
M_d \overset{D}{\to} M,
\end{displaymath}
where $\mathbb{P}( M \leq m) = g_{m+1}(0)$, being $g(s) = \mathbb{E}(s^N) $ and $g_{m+1}(s) = \overset{ m+1 \textrm { times }} {g(g(\cdots g(s)) \cdots )}$.
\end{prop}

\begin{proof}[Proof of Proposition~\ref{P: LIMTSRCR1}]
Every vertex $x$ which is colonized produces
$Y_d$ new colonies (whose distribution depends only on $N$) on the $d$ neighbour vertices which
are located further from the origin than $x$ is. The random variable $Y_d$ can be seen as $Y_d = \sum_{i=1}^{d}I_i$
where for $i=1, \dots, d$
\begin{displaymath}
I_i = \left\{%
\begin{array}{ll}
1, & \hbox{the $i-th$ neighbour of $x$ is colonized} \\
0, & \hbox{else.} \\
\end{array}%
\right.
\end{displaymath}
Defining $g_{Y_d}(s)$ as the generating function of $Y_d$ observe that equation~(\ref{E: MediaYL}) gives $g_{Y_d}^{\prime}(1)$. Moreover
\begin{displaymath}
 {Y_d}^2 = \left (\sum_{i=1}^{d}I_i \right )^2 = \sum_{i=1}^{d}I_i^2 + 2\sum_{1 \leq i < j \leq d}I_iI_j
\end{displaymath}
and
\begin{displaymath}
 \mathbb{E} \left({Y_d}^2 \right ) =  d \mathbb{E} \left (I_1^2 \right ) + d(d-1) \mathbb{E}(I_1I_2)
\end{displaymath}
and finally
\begin{displaymath}
 \mathbb{E} \left({Y_d}^2 \right ) =  d \left[ 1 - \mathbb{E} \left [\left( \frac{d}{d+1} \right)^N \right] \right]   + d(d-1) \left [  1 -2  \mathbb{E} \left [\left ( \frac{d}{d+1} \right )^N \right] + \mathbb{E} \left [\left ( \frac{d-1}{d+1} \right )^N \right] \right ]
\end{displaymath}
Then
\begin{align*}
g_{Y_d}^{\prime \prime}(1) = \mathbb{E} \left(Y_d(Y_d-1) \right ) =  d(d-1) \left [  1 -2  \mathbb{E} \left [\left ( \frac{d}{d+1} \right )^N \right] + \mathbb{E} \left [\left ( \frac{d-1}{d+1} \right )^N \right] \right ]
\end{align*}
Then the result follows from Theorem 1 page 331 in~\cite{AA}, where $ m = g_{Y_d}^{\prime}(1)$.

The convergence $M_d \overset{D}{\to} M$ follows from the fact that $Y_d \overset{D}{\to} N$ when $ d \to \infty$ and from Proposition \ref{P: CFGP}.
\end{proof}

\subsection{Proofs of the main results}
\hfill

First we define a coupling between the processes $\procT$ and $\lowP$ in such a way that
the former is dominated by the earlier. Every colony in $\lowP$ is associated to a colony
in $\procT$. As a consequence, if the process $\procT$ dies out, the same happens to
$\lowP$.

At every collapse time at a vertex $x$ in the original model, a non-empty group of individuals
that tries to colonize the neighbour vertex to $x$ which is closer to the origin than $x$ will create there a new colony provided that that vertex is empty. In the  model $\lowP$ the same non-empty group of individuals that tries to colonize the same vertex, imediately dies.

Next we define a coupling between the processes $\procT$ and $\upPd$ in such a way that
the former dominates the earlier. Every colony in $\procT$ can be associated to a colony
in $\upPd$. As a consequence if the process $\upPd$ dies out, the same happens to
$\procT$.

At every collapse time at a vertex $x$ we associate the neighbour
vertex to $x$ which is closer to the origin than $x$ to the extra vertex
on the model $\upPd$. In the original model, a non-empty group of individuals
that tries to colonize the neighbour vertex to $x$ which is closer to the origin than $x$ will create there a new colony provided that that vertex is empty. In the
model $\upPd$ the same non-empty group of individuals that tries to colonize the
extra vertex, founds a new colonony there.

\begin{proof}[Proof of Theorem~\ref{T: MCC1H}]

The result follows from the fact that the process $\procT$ dominates the process
$\lowP$ and by its turn, is dominated by the process $\upPd$, together with
Propositions~\ref{P: CCSRSR3} and~\ref{P: CCSRCR3}.

\end{proof}

\begin{proof}[Proof of Corollary~\ref{C: MCC1HP}]
Assuming $s=\frac{d}{d+1}$ in (\ref{eq: fgpP}) and applying Theorem~\ref{T: MCC1H} the
result follows.
\end{proof}

\begin{proof}[Proof of Corollary~\ref{C: MCC1HPY}]
Assuming $s=\frac{d}{d+1}$ in (\ref{eq: fgpY}) and applying Theorem~\ref{T: MCC1H} the
result follows.
\end{proof}

\begin{proof}[Proof of Theorem~\ref{T: MCCHPE1}]
The result follows from the fact that the process $\procT$ dominates the process
$\lowP$ and by its turn, is dominated by the process $\upPd$, together with
Propositions~\ref{P: CCSRSR1V} and~\ref{P: CCSRCR1V}.
\end{proof}

\begin{proof}[Proof of Theorem~\ref{T: MCCHPE1L}]
The result follows from the fact that the process $\procT$ dominates the process
$\lowP$ and by its turn, is dominated by the process $\upPd$, together with
Propositions~\ref{P: CCSRSR1VL} and~\ref{P: RCR1VL2}.
\end{proof}

\begin{proof}[Proof of Corollary~\ref{C: MCC1HPS}]
The proof is just a matter of computing the smallest positive fixed point for the generating function of
$N$ (the smallest positive $s$ such that $\mathbb{E}(s^N) = s$) for $\mathbb{E}(s^N)$ given in~(\ref{eq: fgpP}).
\end{proof}

\begin{proof}[Proof of Corollary~\ref{C: MCC1HPS2}]
The proof is just a matter of computing the smallest positive fixed point for the generating function of
$N$ (the smallest positive $s$ such that $\mathbb{E}(s^N) = s$) for $\mathbb{E}(s^N)$ given in~(\ref{eq: fgpY}).
\end{proof}

\begin{proof}[Proof of Theorem~\ref{T: LEIT}]
The result follows from the fact that the process $\procT$ dominates the process
$\lowP$ and by its turn, is dominated by the process $\upPd$, together with
Propositions~\ref{P: LIMTSRSR1} and~\ref{P: LIMTSRCR1}.
\end{proof}

\begin{proof}[Proof of Corollary~\ref{C: CTM1}]
The proof is just a matter of computing the generating function of $N$ (see Equation~(\ref{eq: fgpP})) on
both values $s = \frac{d}{d+1}$ and $s = \frac{d-1}{d+1}$.
\end{proof}

\begin{proof}[Proof of Theorem~\ref{T: LEITL}]
The result follows from the fact that the process $\procT$ dominates the process
$\lowP$ and by its turn, is dominated by the process $\upPd$, together with
Propositions~\ref{P: LIMTSRSR1} and~\ref{P: LIMTSRCR1}.
\end{proof}

\begin{defn} \label{fgpfl}
A fractional linear generating function is a probability generating function of the form
\begin{equation*}
f(b,c;s) = 1 - \frac{b}{1-c} + \frac{bs}{1-cs}, \ 0 \leq s \leq 1.
\end{equation*}
where $0 \leq b \leq 1$, $0 \leq c \leq 1$, and $b+c \leq 1$.
\end{defn}

\begin{proof}[Proof of Corollary~\ref{C: LEITL}]
Observe that the generating function of $N$ given in (\ref{eq: fgpP}) is a fractional linear generating function. The results
follow from equations (3.1) and (3.2) in \cite{AA} page 330 and from Theorem \ref{T: LEITL}.
\end{proof}

\begin{proof}[Proof of Theorem~\ref{T: NC}]
The result follows from the fact that the process $\procT$ dominates the process
$\lowP$ and by its turn, is dominated by the process $\upPd$, together with
Propositions~\ref{P: CCSRSR1V} and~\ref{P: CCSRCR1V}.
\end{proof}

\end{document}